\documentclass[a4paper,11pt]{article}

\usepackage[utf8]{inputenc}
\usepackage[T1]{fontenc}
\usepackage{microtype}
\usepackage{bm}

\usepackage[british]{babel}
\usepackage{csquotes}                   

\usepackage[top=3cm, bottom=3cm, left=3cm, right=3cm]{geometry}

\usepackage{fancyhdr}
\usepackage{fancyvrb}

\usepackage{amsmath}
\usepackage{amssymb}
\usepackage{amsthm}
\usepackage{mathtools}

\usepackage{tikz}
\usepackage{enumitem}

\setlist[enumerate,1]{label={\textnormal{(\arabic*)}}}

\usepackage{hyperref}
 \hypersetup{citecolor=purple, linkcolor=blue, colorlinks=true}
\usepackage{cleveref}
\usepackage{bookmark}                   

\usepackage[backend=biber,style=numeric-comp,maxnames=3,doi=false,isbn=false,url=false,eprint=false]{biblatex}
\renewbibmacro{in:}{\ifentrytype{article}{}{\printtext{\bibstring{in}\intitlepunct}}}

\newcommand{\parag}{\bigskip}

\newcommand{\mC}{\mathcal{C}}
\newcommand{\mH}{\mathcal{H}}
\newcommand{\mK}{\mathcal{K}}
\newcommand{\mL}{\mathcal{L}}
\newcommand{\mO}{\mathcal{O}}
\newcommand{\mP}{\mathcal{P}}
\newcommand{\mS}{\mathcal{S}}
\newcommand{\mJ}{\mathcal{J}}

\newcommand{\br}[1]{\left(#1\right)}
\newcommand{\brArg}[1]{\!\br{#1}}

\newcommand{\set}[1]{\left\{#1\right\}}                     
\newcommand{\sett}[2]{\left\{#1\,:\,#2\right\}}             
\newcommand{\vspan}[1]{\left\langle#1\right\rangle}         

\newcommand{\restr}[2]{{#1}\rvert_{#2}}                     

\newcommand{\NN}{\mathbb{N}}
\newcommand{\NNpos}{\NN\setminus\set{0}}
\newcommand{\NNnot}{\NN\setminus\set{0,1}}
\newcommand{\FF}[1]{\mathbb{F}_{#1}}

\newcommand{\zero}{\bm{0}}

\newcommand{\n}{n}

\newcommand{\pg}[2]{\mathrm{PG}\brArg{#1,#2}}    

\newcommand{\code}[3]{\mC_{#1}\brArg{#2,#3}}

\newcommand{\wt}[1]{\mathrm{wt}\brArg{#1}}
\newcommand{\wtInline}[1]{\mathrm{wt}\br{#1}}
\newcommand{\supp}[1]{\mathrm{supp}\brArg{#1}}
\newcommand{\suppp}[2]{\mathrm{supp}_{#1}\brArg{#2}}

\newcommand{\deltta}{\delta}
\newcommand{\Deltta}{\Delta_q}

\newcommand{\proj}[2]{\mathrm{proj}_{#1,#2}}
\newcommand{\projArg}[3]{\proj{#1}{#2}\brArg{#3}}
\newcommand{\prj}[3]{\mathrm{proj}^{(#1)}_{#2,#3}}

\newcommand{\size}[1]{\left|#1\right|}
\newcommand{\floor}[1]{\left\lfloor#1\right\rfloor}
\newcommand{\ceil}[1]{\left\lceil#1\right\rceil}

\renewcommand{\geq}{\geqslant}
\renewcommand{\leq}{\leqslant}

\newtheorem{thm}{Theorem}[section]
\newtheorem{res}[thm]{Result}
\newtheorem{prp}[thm]{Proposition}
\newtheorem{lem}[thm]{Lemma}
\newtheorem{crl}[thm]{Corollary}

\theoremstyle{definition}
\newtheorem{dfn}[thm]{Definition}
\newtheorem{rmk}[thm]{Remark}

\usepackage{devanagari}

\newcommand{\la}{\textnormal{\dn l}}

\newcommand{\mG}{\mathcal G}
\newcommand{\gauss}[2]{\genfrac{[}{]}{0pt}{}{#1}{#2}_q}

\title{Small weight codewords of projective geometric codes II}
\author{Sam Adriaensen \\  {\it Vrije Universiteit Brussel} \and Lins Denaux \\ {\it Ghent University}}
\date{}

\begin{document}

\maketitle

\begin{abstract}
    The \(p\)-ary linear code \(\code{k}{\n}{q}\) is defined as the row space of the incidence matrix \(A\) of \(k\)-spaces and points of \(\pg{\n}{q}\).
    It is known that if \(q\) is square, a codeword of weight \(q^k\sqrt{q}+\mO\brArg{q^{k-1}}\) exists that cannot be written as a linear combination of at most \(\sqrt{q}\) rows of \(A\).
    Over the past few decades, researchers have put a lot of effort towards proving that any codeword of smaller weight \emph{does} meet this property.
    We show that if \(q\geq32\) is a composite prime power, every codeword of \(\code{k}{\n}{q}\) up to weight \(\mO\brArg{q^k\sqrt{q}}\) is a linear combination of at most \(\sqrt{q}\) rows of \(A\).
    We also generalise this result to the codes \(\code{j,k}nq\), which are defined as the \(p\)-ary row span of the incidence matrix of $k$-spaces and $j$-spaces, $j < k$.
\end{abstract}



\section{Introduction and overview}

In this article, we are interested in a particular class of linear codes, which can be defined as follows.
Consider a prime power $q\coloneqq p^h$, with $p$ prime.
Choose integers $0 \leq j < k < n$.
Let $\pg n q$ denote the $n$-dimensional projective space over $\FF q$, and let $\mJ$ and $\mK$ denote the respective sets of $j$-spaces and $k$-spaces of $\pg n q$.
The incidence matrix of $k$- and $j$-spaces is the matrix $A$ whose rows and columns are indexed by $\mK$ and $\mJ$ respectively, which contains a 1 in positions where the corresponding subspaces are incident, and a 0 in all other positions.
Symbolically,
\begin{align*}
 A \in \{0,1\}^{\mK \times \mJ}, &&
 A(\kappa,\lambda) \coloneqq \begin{cases}
  1 & \text{if } \lambda \subset \kappa, \\
  0 & \text{otherwise.}
 \end{cases}
\end{align*}
The codes we are interested in are the row spaces of these incidence matrices.
These codes consist of vectors whose positions are labelled by the $j$-spaces of $\pg n q$.
It is therefore more convenient to interpret the codewords as functions from $\mJ$ to $\{0,1\}$.

\begin{dfn}
For every $k$-space $\kappa$ of $\pg n q$, define its \emph{characteristic function} with respect to the $j$-spaces as the function
\[
 \kappa^{(j)}: \mJ \to \{0,1\}:
 \lambda \mapsto \begin{cases}
     1 & \text{if } \lambda \subset \kappa, \\
     0 & \text{otherwise.}
 \end{cases}
\]
\end{dfn}

Since $\kappa^{(j)}$ only takes the values $0$ and $1$, we can interpret it as a function from $\mJ$ to any field.
We will study these characteristic functions as functions $\mJ \to \FF p$.
The vector space consisting of all functions from $\mJ$ to $\FF p$ will be denoted by $\FF p^\mJ$.

\begin{dfn} \label{Def:TheCode}
 The code $\code {j,k} n q$ is the vector subspace of $\FF p^\mJ$ generated by the set $\sett{\kappa^{(j)}}{\kappa \in \mK}$ of characteristic functions of the $k$-spaces of $\pg n q$ with respect to the $j$-spaces.
 In case $j=0$, we denote these codes by $\code k n q$.
\end{dfn}

We aim to characterise the small weight codewords of $\code {j,k} n q$.
One way to make codewords of relatively small weight is by taking linear combinations of a small number of characteristic functions.
We will say that a codeword $c \in \code {j,k} n q$ is a ``linear combination of (exactly) $m$ $k$-spaces'' if it can be written as a linear combination of $m$ characteristic functions of $k$-spaces, each occurring in the linear combination with a non-zero scalar.
We remark that the characteristic functions (when seen as $p$-ary functions) are linearly dependent, hence if $c \in \code {j,k} n q$, there is not a \emph{unique} linear combination of characteristic functions of $k$-spaces equal to $c$.

\parag
We will begin with an overview of the known results.
Most of the notation and terminology are standard. Everything will be defined in \Cref{Sec:Prel}.

As a first step, the codewords of minimum weight have been characterised.

\begin{res}[{\autocite{AssmusKey,DelsarteGoethalsMacWilliams} and \autocite[Theorem 1]{BagchiInamdar}}]
    The minimum weight of \(\code{j,k}{\n}{q}\) is \(\gauss{k+1}{j+1}\).
    Every minimum weight codeword is a scalar multiple of the characteristic function of a \(k\)-space.
\end{res}

Stronger characterisations are known.

\subsection{The planar case}

Initially, most attention was paid to the smallest set of parameters, i.e.\ the codes $\code 1 2 q$.
Several results emerged in case \(q=p\) is prime, starting with \citeauthor{McGuireWard} \autocite{McGuireWard}.
They discovered a gap in the weight spectrum by proving that no codeword of \(\code{1}{2}{p}\) has weight \(w\in\set{p+2,\dots,\frac{3}{2}(p+1)}\), \(p\neq2\) \autocite[Corollary \(2.3\)]{McGuireWard}.
\citeauthor{Chouinard:PhD} \autocite[Proposition \(27\)]{Chouinard:PhD} extended this result by showing that no codeword has weight \(w\in\set{p+2,\dots,2p-1}\).

A decade later, \citeauthor{FackFancsaliStormeVandeVoordeWinne} \autocite{FackFancsaliStormeVandeVoordeWinne} generalised this result by proving that if \(p\geq11\), any codeword of \(\code{1}{2}{p}\) of weight smaller than \(\frac{5}{2}p\) is equal to a linear combination of at most two lines.
Add another decade, \citeauthor{Bagchi:FourthWeight} \autocite{Bagchi:FourthWeight} extended this result to all codewords of weight smaller than \(3p-3\), \(p\geq5\).

\parag
Generally, researchers try to prove that any codeword \(c\in\code{1}{2}{q}\) whose weight is upper bounded by some function \(W\brArg{q}\) is a linear combination of exactly \(\ceil{\frac{\wtInline{c}}{q+1}}\) lines, which are relatively few.

In \citeyear{Key}, \citeauthor{Key} \autocite{Key} proved that the characteristic function of a Hermitian variety\footnote{
 For any set $\mS$ of points in $\pg n q$, we can define its characteristic function $v_\mS$ as the function that maps the points of $\mS$ to 1, and the other points of $\pg n q$ to 0.
}
is a codeword of \(\code{\n-1}{\n}{q^2}\), while \citeauthor{BlokhuisBrouwerWilbrink} \autocite{BlokhuisBrouwerWilbrink} showed that any unital \(\mH\) of \(\pg{2}{q^2}\) is a non-singular Hermitian curve if and only if its characteristic function \(v_\mH\) is a codeword of \(\code{1}{2}{q^2}\), or, in other words, if and only if \(v_\mH\) is equal to a \(p\)-ary linear combination of characteristic functions of lines.
One can easily prove that any linear combination of lines equal to \(v_\mH\) must consist of at least \(q^2-q+1\) lines, which is substantially larger than \(\ceil{\frac{\wtInline{v_\mH}}{q^2+1}}=q\) and implies that \(W\brArg{q}\) cannot be larger than \(q\sqrt{q}\) if \(q\) is square.

\citeauthor{Bagchi:OddCodeword} \autocite[Theorem \(5.2\)]{Bagchi:OddCodeword} and \citeauthor{DeBoeckVandendriessche} \autocite[Example \(10.3.4\)]{DeBoeck:PhD}, \autocite[Example \(1.8\)]{DeBoeckVandendriessche} independently discovered a peculiar codeword \(c \in \code{1}{2}{p}\) of weight \(3p-3\) that cannot be written as a linear combination of fewer than \(p-1\) lines.
If \(p>3\), then \(p-1\) is larger than \(\ceil{\frac{\wtInline{c}}{p+1}}\leq3\), implying that \(W\brArg{p}\) is at most \(3p-3\) if \(p\geq5\) is prime.

\parag
Using polynomial methods, \citeauthor{SzonyiWeiner} contributed considerably to the characterisation of small weight codewords of \(\code{1}{2}{q}\) for somewhat larger values of \(q\).

\begin{res}[{\autocite[Theorems \(4.3\), \(4.8\) and Corollary \(4.10\)]{SzonyiWeiner}}]\label{res:SzonyiWeiner}
    Let \(c\) be a codeword of \(\code{1}{2}{q}\), \(q=p^h\), \(p\) prime.
    \begin{itemize}
        \item If \(h=1\), \(p\geq19\) and \(\wt{c}\leq\max\set{3p+1,4p-22}\), then \(c\) is either a linear combination of at most three lines or a certain generalisation of the peculiar codeword described above.
        \item If \(h\geq2\), \(q\geq32\) and
        \[
            \wt{c}<\begin{cases}
                \frac{\br{p-1}\br{p-4}\br{p^2+1}}{2p-1}&\text{if}\ h=2,\\
                \br{\floor{\sqrt{q}}+1}\br{q+1-\floor{\sqrt{q}}}&\text{if}\ h\geq3,
            \end{cases}
        \]
        then \(c\) is a linear combination of exactly \(\ceil{\frac{\wtInline{c}}{q+1}}\) lines.
    \end{itemize}
\end{res}

Hence, if \(q\) is neither small nor prime, the above result characterises all codewords of \(\code{1}{2}{q}\) up to weight \(W\brArg{q}=\mO\brArg{q\sqrt{q}}\).
If \(q\geq32\) and if \(h\geq4\) is even, then the result is sharp, as illustrated by the characteristic functions of the Hermitian curves.

\subsection{The general case}

Consider a codeword \(c\in\code{1}{2}{q}\) and embed \(\pg{2}{q}\) as a plane \(\pi\) in \(\pg{\n}{q}\).
By fixing a \(\br{k-2}\)-space \(\Pi\) disjoint to \(\pi\), one can construct from \(c=\sum_{i} \alpha_i \ell_i^{(0)} \) a codeword \( c' \coloneqq \sum_i \alpha_i \vspan{\ell_i,\Pi}^{(0)} \in \code{k}{\n}{q}\) of weight \(\wt{c}q^{k-1}\) or \(\wt{c}q^{k-1}+\theta_{k-2}\), depending on whether or not \(\Pi\) avoids \(\supp{c}\), or equivalently whether or not $\sum_i \alpha_i = 0$.

Therefore, the observations made in the planar case can be related to the more general case of the codes $\code k n q$.
One commonly tries to prove that any codeword \(c\in\code{k}{\n}{q}\) of weight at most (some function) \(W\brArg{k,\n,q}\) is equal to a linear combination of exactly \(\ceil{\frac{\wtInline{c}}{\theta_k}}\) \(k\)-spaces.
Moreover, \(W\brArg{k,\n,q}\) must be smaller than \(q^k\sqrt{q}+\theta_{k-1}\) if \(q\) is square, and smaller than \(\br{3q-3}q^{k-1}\) if \(q\geq5\) is prime.

\parag
Characterising small weight codewords of \(\code{k}{\n}{q}\), \(\n\geq3\), has gained some popularity in recent years.
We present a short overview based on the survey article of \citeauthor{LavrauwStormeVandeVoorde:LinearCodes} \autocite{LavrauwStormeVandeVoorde:LinearCodes}.

\parag
While this was already utilized in the planar case, \citeauthor{LavrauwStormeVandeVoorde:PointsHyperplanes} \autocite{LavrauwStormeVandeVoorde:PointsHyperplanes,LavrauwStormeVandeVoorde:PointsKspaces} exploited a strong link between codewords of \(\code{k}{\n}{q}\) of small weight and blocking sets.
One year later, \citeauthor{LavrauwStormeSziklaiVandeVoorde} \autocite[Theorem \(12\)]{LavrauwStormeSziklaiVandeVoorde} proved that there exist no codewords in \(\code{k}{\n}{q}\setminus\code{\n-k}{\n}{q}^\perp\), \(p>5\), with weight in the interval \(\left]\theta_k,2q^k\right[\).
As pointed out in \autocite[Theorem \(3.12\)]{LavrauwStormeVandeVoorde:LinearCodes}, using a known lower bound on the minimum weight of \(\code{\n-k}{\n}{q}^\perp\) \autocite[Theorem \(3\)]{BagchiInamdar}, one can show that there exist no codewords of \(\code{k}{\n}{q}\), \(p>5\), having weight in the interval \(\left]\theta_k,2\br{\frac{q^\n-1}{q^{\n-k}-1}\br{1-\frac{1}{p}}+\frac{1}{p}}\right[\).

\parag
By analysing what is known about the codewords in \(\code{k}{\n}{q}\cap\code{\n-k}{\n}{q}^\perp\) and narrowing their view to the cases \(k=\n-1\) and \(q\) prime, \citeauthor{LavrauwStormeSziklaiVandeVoorde} managed to prove that no codewords of \(\code{k}{\n}{q}\), \(p>5\), have weight in the interval \(\left]\theta_k,2q^k\right[\) if \(k=\n-1\) or if \(q\) is prime \autocite[Corollaries \(19\) and \(21\)]{LavrauwStormeSziklaiVandeVoorde}.

\parag
Roughly a decade later, \citeauthor{PolverinoZullo:Codes} characterised all codewords of \(\code{\n-1}{\n}{q}\) up to the second smallest non-zero weight:

\begin{res}[{\autocite[Theorem \(1.4\)]{PolverinoZullo:Codes}}]\label{res:PolverinoZullo:Codes}
    There are no codewords of \(\code{\n-1}{\n}{q}\) with weight in the interval \(\left]\theta_{\n-1},2q^{\n-1}\right[\).
    Any codeword of weight \(2q^{\n-1}\) is a non-zero scalar multiple of the difference of two distinct hyperplanes.
\end{res}

For a shorter and self-contained proof of the above result, see \cite[Theorem 4.4]{Adriaensen}.

\parag
The authors of this paper together with Storme and Weiner \autocite{AdriaensenDenauxStormeWeiner} extended \Cref{res:PolverinoZullo:Codes} by proving that all codewords of \(\code{\n-1}{\n}{q}\) up to weight roughly \(4q^{\n-1}\) are linear combinations of hyperplanes through a fixed \(\br{\n-3}\)-space if \(q\) is large enough, which in turn has been improved slightly in \autocite{Denaux:PhD}.
One year later, we \autocite{AdriaensenDenaux} characterised all codewords of \(\code{k}{\n}{q}\), \(q\) large enough, up to weight roughly \( 3q^{k}\) as being linear combinations of at most two \(k\)-spaces.
In addition, we proved a similar result for the more general family of codes arising from the incidence of \(j\)- and \(k\)-spaces.

\parag
Finally, the second author and Bartoli \autocite{BartoliDenaux} showed that if \(q\) is not prime and large enough, then codewords of \(\code{\n-1}{\n}{q}\) up to weight roughly \(\frac{1}{2^{\n-2}}q^{\n-1}\sqrt{q}\) are linear combinations of exactly \(\ceil{\frac{\wtInline{c}}{\theta_{\n-1}}}\) hyperplanes.
One of the aims of this paper is to remove the exponential factor $\frac 1 {2^{n-2}}$.

\subsection{Outline and main result}

In this paper, we prove the following theorem.

\begin{thm}
 \label{thm:GeneralCase}
 Suppose that \(j,k,\n\in\NN\), \(0\leq j < k < \n \), and let \(q\coloneqq p^h\geq32\) with \(p\) prime and \(h\in\NNnot\).
    Consider a codeword \(c\in\code{j,k}{\n}{q}\) with \(\wt{c}\leq\Deltta \gauss{k+1}{j+1}\), where
    \[
        \Deltta\coloneqq\begin{cases}
            \frac{1}{2}\sqrt{q}-\frac{7}{2}&\text{if}\ h=2,\\
            \floor{\sqrt{q}-\frac{3}{2}}&\text{otherwise}.
        \end{cases}
    \]
    Then \(c\) is a linear combination of exactly \(\ceil{\frac{\wtInline{c}}{\gauss{k+1}{j+1}}}\) \(k\)-spaces.
\end{thm}

The paper is structured as follows.
In \Cref{Sec:Prel} we give the necessary definitions and background.
In \Cref{Sec:Spectrum} we prove a crucial intermediary result.
This result roughly states that if a point set in $\pg n q$ intersects all spaces of some fixed dimension in either few or many points, then the point set is either small or large.
After this, we are ready to start the proof of \Cref{thm:GeneralCase}.
This is done by using induction on each of the parameters $n$, $k$ and $j$.
In \Cref{Sec:PtHyp} we prove the theorem for the codes $\code {n-1} n q$, in \Cref{Sec:PtK} we prove it for the codes $\code k n q$ and in \Cref{Sec:j&k}, we finish the proof for the general case.

\section{Preliminaries} \label{Sec:Prel}

\subsection{Finite projective geometries}

Throughout this work, we assume that \(\n\in\NNnot\) and that \(q\) is a prime power, i.e.\ \(q\coloneqq p^h\), where \(p\) is prime and \(h\in\NNpos\).
We will mostly consider the case \(h\geq2\).
Finally, we assume that \(j\) and \(k\) are integers satisfying \(0 \leq j < k < n \).

\parag
The Galois field of order \(q\) will be denoted by \(\FF{q}\) and the Desarguesian projective geometry of (projective) dimension \(\n\) over \(\FF{q}\) will be denoted by \(\pg{\n}{q}\).
Whenever `dimension' or `(sub)space' is mentioned, these are implied to be \emph{projective}.
When working in $\pg nq$, we denote the set of $j$-spaces incident with a given subspace $\kappa$ by $\mG_j(\kappa)$.
The set of all $j$-spaces is denoted by $\mG_j$, or $\mG_j(n,q)$ if we want to emphasise the ambient projective geometry.
The number of $k$-spaces through a fixed $j$-space in $\pg nq$ is given by the Gaussian coefficient
\[
    \gauss{n-j}{k-j} \coloneqq \prod_{i=1}^{k-j} \frac{q^{n-k+i}-1}{q^i-1}.
\]
For simplicity's sake, we denote the number of points (or hyperplanes) of \(\pg{\n}{q}\) by \(\theta_{\n}\), i.e.
\[
    \theta_{\n} \coloneqq \gauss{n+1}1 = \frac{q^{\n+1}-1}{q-1}=q^\n+q^{\n-1}+\dots+q+1,
\]
where we settle on the convention that \(\theta_{-1}\coloneqq0\).

\begin{dfn}
 Let $\mS$ be a set of points in $\pg nq$.
 We say that $\mS$ is a \emph{blocking set with respect to the $k$-spaces} or that $\mS$ \emph{blocks all $k$-spaces} if it intersects every $k$-space.
\end{dfn}

A famous result by Bose and Burton gives a lower bound on the size of a blocking set.

\begin{res}[{\cite{boseburton}}]
 \label{Res:BlockingSet}
 If $\mS$ blocks all $k$-spaces, then $|\mS| \geq \theta_{n-k}$, and equality occurs if and only if $\mS$ is an $(n-k)$-space.
\end{res}

\subsection{Codes from projective geometries}

As mentioned in the introduction, we are interested in the codes $\code {j,k} n q$ (see \Cref{Def:TheCode}), whose ambient vector space is $\FF p^{\mG_j}$.
We define the \emph{support} of \(v \in \FF p^{\mG_j} \) to be the set
\[
    \supp v \coloneqq \sett{\lambda \in \mG_j}{v(\lambda) \neq 0}.
\]
More generally, we define, for each $i\in\set{0,1,\dots,j}$, the set
\[
    \suppp i v \coloneqq \sett{ \iota \in \mG_i}{(\exists \, \lambda \in \supp v)(\iota \subseteq \lambda)}.
\]
In case $j=0$, points having value \(0\) with respect to \(v\) are called \emph{holes} with respect to \(v\).
The \emph{weight} \(\wt{v}\) of \(v\) is equal to the size of its support, i.e.\ \(\wt{v}\coloneqq\size{\supp{v}}\).

\begin{prp}\label{prp:Intuition}
    If \(c\in\code{j,k}{\n}{q}\) is a linear combination of exactly \(m \leq\sqrt{q^{j+1}}\) \(k\)-spaces, then
    \begin{enumerate}
        \item \( \displaystyle m = \ceil{ \wt{c} / \gauss{k+1}{j+1}} \),
        \item if $j=0$, every subspace contains either at most \( m \) or at least \(q-m+2\) points of \(\supp{c}\).
    \end{enumerate}
\end{prp}

\begin{proof}
    (1) Since any two $k$-spaces share at most $\gauss k{j+1}$ $j$-spaces, we know that
    \begin{align*}
        m \gauss{k+1}{j+1} \geq \wt c
        & \geq m \left( \gauss{k+1}{j+1} - (m-1) \gauss{k}{j+1} \right) > m \gauss{k+1}{j+1} - q^{j+1} \gauss k{j+1} \\
        & = \left( m - q^{j+1} \frac{q^{k-j}-1}{q^{k+1}-1} \right) \gauss{k+1}{j+1} > (m-1) \gauss{k+1}{j+1}.
    \end{align*}
    
    (2) We prove this statement by induction on $m$.
    Note that it is trivial for $m=0$.
    Now assume that the statement holds for all $m' < m$.
    Suppose that
    \(
        c = \sum_{i=1}^m \alpha_i \kappa_i^{(0)},
    \)
    with all $\kappa_i$ distinct $k$-spaces and $\alpha_i \in \FF p^*$.
    Let $\rho$ be a subspace and let $\sigma$ be an element of $\sett{\kappa_i \cap \rho}{i=1,\dots,m}$ of maximal dimension $s$.
    If $s \leq 0$, then $\rho$ trivially contains at most $m$ points of $\supp c$, so we only need to consider the case where $s \geq 1$.
    Define the set $I \coloneqq \sett{i \in \{1,\dots,m\}}{\sigma \subseteq \kappa_i}$.
    
    First, suppose that
    \(
        \sum_{i \in I} \alpha_i = 0.
    \)
    Then $\supp c \cap \rho = \supp {c'} \cap \rho$, with
    \(
        c' \coloneqq \sum_{i \notin I} \alpha_i \kappa_i^{(0)}.
    \)
    Since $c'$ is a linear combination of less than $m$ $k$-spaces, the statement follows from the induction hypothesis.
    
    Next, suppose that $\sum_{i \in I} \alpha_i \neq 0$.
    Then all points of $\sigma \setminus \bigcup_{i \notin I} \kappa_i$ have the same non-zero coefficient with respect to $c$.
    It follows that
    \[
        \size{\supp c \cap \rho}
        \geq \size{\sigma \setminus \bigcup_{i \notin I} \kappa_i}
        \geq \theta_s - (m-1) \theta_{s-1}
        = (q-m+1) \theta_{s-1} + 1
        \geq q-m+2. \qedhere
    \]
\end{proof}

For any \(i\)-space \(\iota\) of \(\pg{\n}{q}\), we can naturally define the \emph{restriction} of \(v\in\FF{p}^{\mG_j(n,q)}\) to \(\iota\) as the function \(\restr{v}{\iota}\in\FF{p}^{\mG_j(\iota)}\) by restricting the domain of $v$ to $\mG_j(\iota)$.
Using the fact that scalar multiples of the all-one function are codewords of \(\code{\n-1}{\n}{q}\), the following can easily be proved.

\begin{res}[{\autocite[Remark \(3.1\)]{PolverinoZullo:Codes}}]\label{res:Restricted}
    Suppose that \(c\in\code{\n-1}{\n}{q}\) and let \(\iota\) be an \(i\)-space of \(\pg{\n}{q}\).
    Then \(\restr{c}{\iota}\in\code{i-1}{i}{q}\).
\end{res}

The following projection map is originally due to \citeauthor{LavrauwStormeVandeVoorde:PointsKspaces} \autocite[Lemma \(11\)]{LavrauwStormeVandeVoorde:PointsKspaces} and was generalised to arbitrary \(j\)-spaces in \autocite{AdriaensenDenaux}.

\begin{dfn}\label{dfn:Proj}
    Let $R$ be a point and $\Pi \not \ni R$ a hyperplane of $\pg nq$.
    Given $v \in \FF p^{\mG_j(n,q)}$, define
    \[
        \prj j R\Pi (v): \mG_j(\Pi) \to \FF p:
        \lambda \mapsto \sum_{\lambda' \in \mG_j(\vspan{R,\lambda})} v(\lambda').
    \]
    Hence, $\prj j R \Pi$ is a map $\FF p^{\mG_j(n,q)} \to \FF p^{\mG_j(\Pi)}$.
    We will denote $\prj 0 R \Pi$ simply by $\proj R \Pi$.
\end{dfn}

\begin{res}[{\autocite[Lemma \(5.2\)]{AdriaensenDenaux}}]\label{res:Projection}
    Assume that \(k\leq\n-2\) and let \(\br{R,\Pi}\) be a non-incident point-hyperplane pair of \(\pg{\n}{q}\).
    \begin{enumerate}
        \item $\prj j R \Pi$ maps $\code {j,k}nq$ to $\code {j,k}{n-1}q$.
        \item If $R \notin \suppp0c$, then $\wt{\prj j R \Pi (c)} \leq \wt c$.
    \end{enumerate}
\end{res}

\subsection{The expander mixing lemma}

We will introduce a helpful tool for counting problems in finite geometry.
This lemma is situated in algebraic graph theory, but we can use it in the context of finite geometry without having to use graph theory terminology.
As far as we are aware, the earliest occurrence of the expander mixing lemma in the form in which we will use it was in the PhD thesis of Haemers, see e.g.\ the summary article of his thesis \cite[Theorem 5.1]{haemers} (including the paragraph after the proof of Theorem 5.1 for the determination of the relevant eigenvalue).
For a statement of the expander mixing lemma more closely resembling the one that we will use, the reader may consult for instance \cite[Lemma 8]{dewinterschillewaertverstraete}.
Recall that a $2-(v,k,\lambda)$ design is an incidence structure consisting of points and blocks, such that
\begin{itemize}
 \item there are $v$ points,
 \item every block contains $k$ points, and
 \item through any two distinct points, there are exactly $\lambda$ blocks.
\end{itemize}
Every point is contained in the same number of blocks $r$, called the replication number of the design.

\begin{lem}[Expander mixing lemma]
 \label{Lem:Expander}
    Consider a $2-(v,k,\lambda)$ design.
    Let $S$ be a set of points and $T$ be a set of blocks. Denote the number of incidences between $S$ and $T$ by
    \[
        e(S,T) \coloneqq \size{\sett{(P,B) \in S \times T}{ P \in B}}.
    \]
    Then
    \[
        \left| e(S,T) - \frac k v \size{S} \size{T} \right| < \sqrt{ (r-\lambda) \size{S} \size{T}}.
    \]
\end{lem}

Remark that if the design consists of the points and $j$-spaces of $\pg nq$, then
\[
 r - \lambda = \gauss n j - \gauss{n-1}{j-1} = q^j \gauss{n-1}j,
\]
hence, in this case, the expander mixing lemma tells us that
\begin{align}
 \label{Eq:EML}
 \left| e(S,T) - \frac {\theta_j}{\theta_n} \size{S} \size{T} \right| < \sqrt{ q^j \gauss{n-1}j \size{S} \size{T}}.
\end{align}

\section{Amplifying a gap in the intersection sizes with subspaces}
 \label{Sec:Spectrum}

In this section, we show that if a point set intersects every $r$-space in either a few or many points, the same should be true for all higher-dimensional subspaces.

\begin{thm}\label{thm:WeightSpectrumSubspaces}
    Consider a prime power $q \geq 16$ and integers $r, n, \delta$ satisfying $1 \leq r < n$ and $\delta \leq \sqrt q - 1$.
    Suppose that $S$ is a set of points in $\pg nq$ intersecting every $r$-space in either
    \begin{align*}
        \text{at most $\delta$ points}  && \text{or} &&  \text{at least $q-\sqrt q +3$ points.}
    \end{align*}
    Then 
    \begin{align*}
        \size{S} \leq \delta \theta_{n-r} && \text{or} && \size{S} > \left( q - \sqrt q + \frac 32 \right) \frac{q^n-1}{q^r-1}.
    \end{align*}
\end{thm}

We will prove this theorem throughout this section.
The main tools are two useful counting techniques.
The first one is sometimes referred to as the standard equations, the second one is the expander mixing lemma.
Although they usually yield the same result in the context of finite geometric counting problems, we will use them here to complement each other.

We make the following conventions for the remainder of this section:
\begin{itemize}
    \item $q \geq 16$,
    \item $1 \leq r < n$ are integers,
    \item $\delta$ is an integer satisfying $\delta \leq \sqrt q - 1$,
    \item $S$ is a set of points in $\pg nq$ intersecting every $r$-space in either at most $\delta$ points or at least $q-\sqrt q+3$ points, and
    \item $s \coloneqq \size{S}$.
\end{itemize}

\begin{lem}
    \label{Lem:ForbiddenSpectrumBulk}
    Either $s < \displaystyle \br{\sqrt q - \frac 12} \frac{q^n-1}{q^r-1}$ or $s > \displaystyle \left( q - \sqrt q + \frac 32 \right) \frac{q^n-1}{q^r-1}$.
\end{lem}

\begin{proof}
    We use the standard equations.
    Let $n_i$ denote the number of $r$-spaces intersecting $S$ in exactly $i$ points.
    It follows from our assumptions that $S$ intersects every $r$-space in at most $\sqrt q - 1$ points or in at least $q-\sqrt q+3$ points.
    Hence,
    \begin{equation}
        \label{Eq:StandardEq}
        \sum_i \left(i-\br{\sqrt q - 1}\right) \left(i-\left(q-\sqrt q + 3\right)\right) n_i \geq 0.
    \end{equation}
    On the other hand, we know that
    \begin{align*}
        \sum_i n_i &= \gauss{n+1}{r+1} = \frac{q^{n+1}-1}{q^{r+1}-1} \frac{q^n-1}{q^r-1} \gauss{n-1}{r-1}, \\
        \sum_i i n_i &= s \gauss{n}{r} = s \frac{q^n-1}{q^{r}-1} \gauss{n-1}{r-1}, \\
        \sum_i i(i-1) n_i &= s(s-1)\gauss{n-1}{r-1}.
    \end{align*}
    The first equation should be clear.
    The second and third equations follow from performing a double count on the following sets, respectively:
    \begin{align*}
        \sett{(P,\rho) \in S \times \mG_{r}}{P \in \rho}, &&
        \sett{(P,R,\rho) \in S \times S \times \mG_{r}}{P \neq R, \, P,R \in \rho}.
    \end{align*}
    Plugging this into \Cref{Eq:StandardEq} and dividing by $\gauss{n-1}{r-1}$ yields
    \begin{align*}
        0 \leq s(s-1) - (q+1) \frac{q^n-1}{q^{r}-1} s + \br{\sqrt q - 1} (q-\sqrt q+3) \frac{q^{n+1}-1}{q^{r+1}-1} \frac{q^n-1}{q^{r}-1}.
    \end{align*}
    Since $s \geq 0$ and $\frac{q^{n+1}-1}{q^{r+1}-1} < \frac{q^n-1}{q^{r}-1}$, we obtain
    \begin{equation}
        \label{Eq:StandardEq2}
        0 < s^2 -  (q+1) \frac{q^n-1}{q^{r}-1} s + \br{\sqrt q - 1} (q-\sqrt q+3) \left( \frac{q^n-1}{q^{r}-1} \right)^2.
    \end{equation}
    The right-hand side of \Cref{Eq:StandardEq2} is a quadratic polynomial in $s$.
    We will give estimates of its roots.
    The discriminant of the polynomial is given by
    \[
        \left( \frac{q^n-1}{q^{r}-1} \right)^2 \left( (q+1)^2 - 4 \br{\sqrt q - 1} (q-\sqrt q+3) \right)
        \geq \left( \frac{q^n-1}{q^{r}-1} (q-2\sqrt q + 2) \right)^2.
    \]
    Therefore, \Cref{Eq:StandardEq2} does not hold when
    \begin{align*}
        s &\in \left[ \frac 12 \frac{q^n-1}{q^{r}-1}( (q+1) - (q-2\sqrt q+2) ), \frac 12 \frac{q^n-1}{q^{r}-1}( (q+1) + (q-2\sqrt q+2) ) \right] \\
        & = \left[ \br{\sqrt q - \frac 12} \frac{q^n-1}{q^{r}-1}, \left( q - \sqrt q + \frac 32 \right) \frac{q^n-1}{q^{r}-1} \right]. \qedhere
    \end{align*}
\end{proof}

It only remains to exclude the case
\[
    \delta \theta_{\n-r} < s < \br{\sqrt q - \frac 12} \frac{q^n-1}{q^{r}-1}.
\]
We will do this by fixing $r$ and using induction on $n$.
Call an $(r+i)$-space $\rho$ \emph{poor} if $\size{\rho \cap S} \leq \delta \theta_i$, and \emph{rich} if $\size{\rho \cap S} \geq \br{q-\sqrt q + \frac 3 2} \frac{q^{r+i}-1}{q^r-1}$.
It follows from our assumptions and the induction hypothesis that, for every $i\in\set{0,1,\dots,n-r-1}$, every $(r+i)$-space $\rho$ is either poor or rich.

\parag
Let $T$ denote the set of rich hyperplanes and define $t \coloneqq \size{T}$.

\begin{lem}
    \label{Lem:LowerBoundOnT}
    Suppose that $s > \delta \theta_{n-r}$.
    Then $t \geq \theta_{r}$.
\end{lem}

\begin{proof}
    Let $\rho$ be an $(r-1)$-space.
    We will prove that $\rho$ lies in a rich hyperplane.
    If $\rho$ only lies in poor $r$-spaces, then
    \[
        s \leq \size{\rho \cap S} + \theta_{n-r} (\delta - \size{\rho \cap S}) \leq \delta \theta_{n-r},
    \]
    contradicting our assumptions.
    Thus, $\rho$ lies in a rich $r$-space.
    
    Now we prove, for every $i \in \set{0,1,\dots,n-r-2}$, that a rich $(r+i)$-space lies in a rich $(r+i+1)$-space.
    Then it follows by induction that $\rho$ lies in a rich hyperplane.
    So suppose, to the contrary, that $\sigma$ is a rich $(r+i)$-space lying only in poor $(r+i+1)$-spaces.
    Then
    \begin{align*}
        \delta \theta_{n-r} 
        < s &\leq \size{\sigma \cap S} + \theta_{n-r-i-1} (\delta \theta_{i+1} - \size{\sigma \cap S}) \\
        &< \delta \theta_{n-r-i-1} \theta_{i+1} - q \theta_{n-r-i-2} \left(q-\sqrt q + \frac 3 2 \right) q^i.
    \end{align*}
    After multiplying both sides by $(q-1)^2$ and moving some terms around, we obtain
    \begin{align*}
        q \left( q^{n-r-i-1}-1 \right) & \left(q-\sqrt q + \frac 3 2 \right) q^i (q-1) \\
        & < \delta \left[ \left( q^{n-r-i}-1 \right) \left( q^{i+2}-1 \right) - \left(q^{n-r+1}-1 \right) (q-1) \right] \\
        & = \delta \left( q^{n-r+1} - q^{n-r-i} - q^{i+2} + q \right) \\
        & = \delta q \left( q^{i+1}-1 \right) \left( q^{n-r-i-1} - 1 \right).
    \end{align*}
    It follows that
    \begin{align*}
        \sqrt q - 1
        & \geq \delta > \frac{\left(q-\sqrt q + \frac 3 2 \right) q^i (q-1)}{q^{i+1}-1} > \left(q-\sqrt q + \frac 3 2 \right) \left( 1 - \frac 1 q \right).
    \end{align*}
    This yields a contradiction for $q \geq 2$.
    
    Thus, we may conclude that every $(r-1)$-space $\rho$ lies in a rich hyperplane.
    This means that any duality of the projective space maps $T$ to a blocking set of the $(n-r)$-spaces.
    Therefore, by \Cref{Res:BlockingSet}, there are at least $\theta_{r}$ rich hyperplanes.
\end{proof}

\begin{lem}
    \label{Lem:UpperBoundOnT}
    Suppose that $s \leq \br{\sqrt q - \frac 12} \frac{q^n-1}{q^{r}-1}$.
    Then
    \[
        t < \frac{\sqrt q}{\left( \sqrt q - 2 \right)^2} q^r.
    \]
\end{lem}

\begin{proof}
    Recall that $e(S,T)$ denotes the number of incidences between the set $S$ of points and the set $T$ of rich hyperplanes.
    By the definition of rich, we have that
    \[
      e(S,T) \geq t \br{q-\sqrt q + \frac 3 2} \frac{q^{n-1}-1}{q^r-1}.
    \]
    On the other hand, by applying the expander mixing lemma to $S$ and $T$, we have that
    \[
        e(S,T)
        \leq \frac{\theta_{n-1}}{\theta_n} st + \sqrt{q^{n-1}st}.
    \]
    Hence,
    \begin{equation}
        \label{Eq:EmlIneq}
        t \left( \br{q-\sqrt q + \frac 3 2} \frac{q^{n-1}-1}{q^r-1} - \frac{\theta_{n-1}}{\theta_n} s \right) \leq \sqrt{q^{n-1} s t}.
    \end{equation}
    Next, we verify that the left-hand side of \Cref{Eq:EmlIneq} is non-negative.
    This follows from
    \begin{align*}
     s \leq \br{\sqrt q - \frac 12} \frac{q^n-1}{q^{r}-1}
     < \sqrt q \frac{q^n}{q^r-1}
     \leq \frac{\theta_n}{\theta_{n-1}} \left( q-\sqrt q + \frac 3 2 \right) \frac{q^{n-1}-1}{q^r-1}.
    \end{align*}
    Hence, we may square both sides of \Cref{Eq:EmlIneq} and the inequality still holds.
    From this we obtain
    \begin{align*}
        t & \leq \frac{ q^{n-1} s}{\left( \br{q-\sqrt q + \frac 3 2} \frac{q^{n-1}-1}{q^r-1} - \frac{\theta_{n-1}}{\theta_n} s \right)^2} \\
        & \leq \frac{
         q^{n-1} \br{\sqrt q - \frac 12} \frac{q^n-1}{q^r-1}
        }{
         \br{ \br{q-\sqrt q + \frac 3 2} \frac{q^{n-1}-1}{q^r-1} - \frac{q^n-1}{q^{n+1}-1} \br{\sqrt q - \frac 12} \frac{q^n-1}{q^r-1} }^2
        } \\
        & = \frac{
         q^{n-1} \br{\sqrt q - \frac 12} \br{q^n-1} \br{q^r-1}
        }{ \br{
         \br{q-\sqrt q + \frac 3 2} \br{q^{n-1}-1} - \frac{q^n-1}{q^{n+1}-1} \br{\sqrt q - \frac 12} \br{q^n-1}
        }^2 } \\
        & < \frac{\sqrt q q^{2n+r-1}}{ \br{ \br{q- 2 \sqrt q} q^{n-1}}^2 }
        = \frac{\sqrt q}{\left( \sqrt q - 2 \right)^2} q^r.
        \qedhere
    \end{align*}
\end{proof}

\begin{lem}
    \label{Lem:ForbiddenSpectrumSmall}
    It is not possible that
    \[
        \delta \theta_{n-r} < s < \left( \sqrt q - \frac 1 2 \right) \frac{q^n-1}{q^{r}-1}.
    \]
\end{lem}

\begin{proof}
    Suppose the contrary.
    By \Cref{Lem:LowerBoundOnT} and \Cref{Lem:UpperBoundOnT},
    \[
        q^{r} < \theta_{r} \leq t < \frac{\sqrt q}{\left( \sqrt q - 2 \right)^2}  q^r.
    \]
    In particular, this implies that
    \[
        \sqrt q > \br{ \sqrt q - 2 }^2
    \]
    which yields a contradiction for $q \geq 16$.
\end{proof}

\Cref{Lem:ForbiddenSpectrumBulk} and \Cref{Lem:ForbiddenSpectrumSmall} together prove \Cref{thm:WeightSpectrumSubspaces}.

\begin{rmk}
 \Cref{thm:WeightSpectrumSubspaces} is a rough generalisation of a theorem from the second author's PhD thesis \autocite[Theorem \(2.2.1\)]{Denaux:PhD}, which considers the case \(r=1\).
\end{rmk}

\section{Codes of points and \texorpdfstring{\(\bm{k}\)}{k}-spaces}

This section is dedicated to proving \Cref{thm:GeneralCase} in case \(j=0\), i.e.\ the following theorem:

\begin{thm}\label{thm:OrdinaryCase}
    Suppose that \(k,\n\in\NN\), \(0<k<\n\), and let \(q\coloneqq p^h\geq32\) with \(p\) prime and \(h\in\NNnot\).
    Consider a codeword \(c\in\code{k}{\n}{q}\) with \(\wt{c}\leq\Deltta\theta_k\), where
    \[
        \Deltta\coloneqq\begin{cases}
            \frac{1}{2}\sqrt{q}-\frac{7}{2}&\text{if}\ h=2,\\
            \floor{\sqrt{q}-\frac{3}{2}}&\text{otherwise}.
        \end{cases}
    \]
    Then \(c\) is a linear combination of exactly \(\ceil{\frac{\wtInline{c}}{\theta_k}}\) \(k\)-spaces.
\end{thm}

Mimicking \Cref{Sec:Spectrum}, the proof will be given throughout this section.
In essence, we use induction on both \(n\) and \(k\), starting from the plane case described in \Cref{res:SzonyiWeiner}.

As \Cref{thm:OrdinaryCase} suggests, we will often make use of the following integer value:
\[
    \Deltta\coloneqq\begin{cases}
        \frac{1}{2}\sqrt{q}-\frac{7}{2}&\text{if}\ q=p^2,\\
        \floor{\sqrt{q}-\frac{3}{2}}&\text{if}\ q=p^h, h\geq3.
    \end{cases}
\]
One can check that
\begin{equation}\label{eq:SzonyiWeiner}
    \br{\Deltta+1}\br{q+1}<\begin{cases}
        \frac{\br{\sqrt{q}-1}\br{\sqrt{q}-4}\br{q+1}}{2\sqrt{q}-1}&\text{if}\ q=p^2,\\
        \br{\floor{\sqrt{q}}+1}\br{q+1-\floor{\sqrt{q}}}&\text{if}\ q=p^h, h\geq3.
    \end{cases}
\end{equation}

Moreover, throughout this section, we
\begin{itemize}
    \item assume that \(q\geq32\) is not prime, and
    \item suppose, to the contrary, that there exists some codeword \(c\in\code{k}{\n}{q}\), \(\wt{c}\leq\Deltta\theta_k\), for which \Cref{thm:OrdinaryCase} is not true.
    We then choose \(\n\), \(k\) and \(\wt{c}\) --- in that order --- to be minimal with respect to this property.
\end{itemize}

\begin{dfn}
    An \(i\)-space \(\iota\) is called \emph{thick} (with respect to \(c\)) in case
    \[
        \wt{\restr{c}{\iota}}\geq\begin{cases}q-\sqrt{q}+3&\text{if}\ i=1,\\\br{q-\sqrt{q}+1}\theta_{i-1}&\text{if}\ i\geq2.\end{cases}
    \]
\end{dfn}

\begin{lem}
 \label{lem:HalfPointTrick}
 There are no thick $k$-spaces.
\end{lem}

\begin{proof}
 Suppose, to the contrary, that $\kappa$ is a thick $k$-space.
 Then $\kappa$ has at most
 \[
  \theta_k - \br{q-\sqrt q + 1} \theta_{k-1}
  = \br{\sqrt q - 1} \theta_{k-1} + 1
 \]
 holes.
 Since there are $p-1 \leq \sqrt q - 1$ scalars in $\FF p^*$, one of these occurs at least
 \[
  \frac{q-\sqrt q + 1}{\sqrt q - 1} \theta_{k-1}
  > \br{\sqrt q - 1} \theta_{k-1} + 1
 \]
 times as the coefficient of a point of $\kappa$ with respect to $c$.
 Call this scalar $\alpha$.
 Define $c' \coloneqq c-\alpha \kappa^{(0)}$.
 Then $\wt{c'} < \wt c$.
 Therefore, due to the minimal weight of \(c\), the codeword \(c'\) must be equal to a linear combination of at most \(\Deltta\) \(k\)-spaces, which means that \(c=c'+\alpha\kappa^{(0)}\) has to be a linear combination of at most \(\Deltta+1\) \(k\)-spaces.
 But then, by \Cref{prp:Intuition}(1), \(c\) is a codeword for which \Cref{thm:OrdinaryCase} is true, a contradiction.
\end{proof}

As \(\wt{c}\leq\Deltta\theta_k<\theta_{k+1}\), there must exist an \(\br{\n-k-1}\)-space \(\rho\) that avoids \(\supp{c}\), cf.\ \Cref{Res:BlockingSet}.
If all \(\br{\n-k}\)-spaces would contain at least \(\Deltta+1\) points of \(\supp{c}\), then so would all \(\br{\n-k}\)-spaces through \(\rho\), implying that \(\wt{c}\geq\br{\Deltta+1}\theta_k\), a contradiction.
Therefore, the following is well-defined.
\begin{dfn}\label{dfn:deltta}
    Define \(\deltta\) to be the largest number in \(\set{0,1,\dots,\Deltta}\) for which there exists an \(\br{\n-k}\)-space containing precisely \(\deltta\) points of \(\supp{c}\).
\end{dfn}

\subsection{The case \texorpdfstring{\(\bm{k=\n-1}\)}{k = n-1}}
 \label{Sec:PtHyp}

Throughout this subsection, we assume that \(k=\n-1\).

\begin{dfn}[\(k=\n-1\)]
    An \(i\)-space \(\iota\) is called \emph{thin} (with respect to \(c\)) in case
    \[
        \wt{\restr{c}{\iota}}\leq\deltta\theta_{i-1}.
    \]
\end{dfn}

\begin{prp}\label{prp:PlaneBound}
    Let \(\pi\) be a plane that contains an \(m\)-secant \(\ell\) to \(\supp{c}\).
    Then
    \[
        \wt{\restr{c}{\pi}}\geq\begin{cases}
            \br{\Deltta+1}\br{q+1}+1&\text{if}\ \Deltta+2\leq m\leq q-\Deltta,\\
            m\br{q-m+2}&\text{otherwise}.
            \end{cases}
    \]
\end{prp}
\begin{proof}
    Suppose that \(\wt{\restr{c}{\pi}}\leq\br{\Deltta+1}\br{q+1}\).
    Then by \Cref{eq:SzonyiWeiner} and Results \ref{res:Restricted} and \ref{res:SzonyiWeiner}, there exists a set \(\mL\) of at most \(\Deltta+1\) lines covering the points of \(\supp{\restr{c}{\pi}}\), each such a line containing at least \(q-\size{\mL}+2\geq q-\Deltta+1\) unique points of \(\supp{\restr{c}{\pi}}\).
    If \(\ell\in\mL\), then \(m\geq q-\Deltta+1\).
    If \(\ell\notin\mL\), then it intersects each line of \(\mL\) in exactly one point, implying that \(m\leq\size{\mL}\leq\Deltta+1\).

    In conclusion, if \(\wt{\restr{c}{\pi}}\leq\br{\Deltta+1}\br{q+1}\), then either \(m\leq\Deltta+1\) or \(m\geq q-\Deltta+1\).
    Moreover, by the above observation, we know that
    \[
        \wt{\restr c \pi} \geq |\mL| \left( q - |\mL| + 2 \right)
     \geq m (q-m+2).\qedhere
    \]
\end{proof}

\begin{lem}\label{lem:ThinOrThickLinesHYP}
    Every line is either thin or thick.
\end{lem}
\begin{proof}
    Consider an arbitrary \(m\)-secant \(\ell\) with respect to \(\supp{c}\) and suppose, to the contrary, that \(\deltta+1\leq m<q-\sqrt{q}+3\), or, equivalently, that \(\Deltta+1\leq m\leq q-\floor{\sqrt{q}}+2\).
    Note that not every plane through $\ell$ contains at least $q \Deltta$ points of $\supp c \setminus \ell$.
    Otherwise,
    \[
     \wt c \geq \br{q \Deltta} \theta_{n-2} + m
     \geq q \Deltta \theta_{n-2} + \Deltta + 1
     = \Deltta \theta_{n-1} + 1,
    \]
    a contradiction.

    First, suppose that \( \Deltta + 2 \leq m \leq q - \Deltta \).
    By \Cref{prp:PlaneBound}, every plane through \(\ell\) has to contain at least 
    \[
     \br{\Deltta+1}\br{q+1}+1 - m \geq \Deltta (q+2) + 2 > q \Deltta
    \]
    points of \(\supp{c} \setminus \ell \).
    This leads to a contradiction.

    Similarly, if \(m = \Deltta+1 \), \Cref{prp:PlaneBound} implies that each plane through \(\ell\) contains at least
    \[
        \br{\Deltta + 1} \br{ q - \Deltta + 1 } - \br{\Deltta + 1} = q \Deltta + q - \Deltta^2 - \Deltta > q \Deltta
    \]
    points of \(\supp{c} \setminus \ell \), again yielding a contradiction.

    Since $\Deltta \leq \floor{\sqrt q - \frac 32} \leq \floor{\sqrt q} - 1$, the only remaining case to exclude is $m = q - \floor{\sqrt q} + 2$ and $\Deltta = \floor{\sqrt q} - 1$.
    In particular, this can only happen if $h \geq 3$.
    Not every plane through $\ell$ contains at least $\br{\floor{\sqrt q} + 1} \br{q-\floor{\sqrt q} + 1}$ points of $\supp{c}$, since
    \begin{align*}
     \br{\floor{\sqrt q} + 1} \br{q-\floor{\sqrt q} + 1} - \br{q - \floor{\sqrt q} + 2}
     &= \floor{\sqrt q} q - \floor{\sqrt q}^2 + \floor{\sqrt q} - 1 \\
     \geq \br{\floor {\sqrt q} - 1} q + \floor{\sqrt q} - 1
     &= q\Deltta + \floor{\sqrt q} - 1
     > q\Deltta.
    \end{align*}
    Thus there exists a plane $\pi$ through $\ell$ with
    \(
     \wt{\restr c \pi} < \br{\floor{\sqrt q} + 1} \br{q-\floor{\sqrt q} + 1}
    \).
    By \Cref{res:SzonyiWeiner}, $\restr c \pi$ is a linear combination of
    \[
     \ceil{\frac{\wt {\restr c \pi}}{q+1}}
     \leq \ceil{ \frac{\br{\floor{\sqrt q} + 1} \br{q-\floor{\sqrt q} + 1}}{q+1} } \leq \floor{\sqrt q} + 1
    \]
    lines.
    Moreover, since $\restr c \pi$ has a $(q-\floor{\sqrt q} + 2)$-secant, it cannot be a linear combination of fewer than $\floor{\sqrt q}$ lines by \Cref{prp:Intuition}(2).
    Since $\binom{\floor{\sqrt q} + 1}2 < q+1$, the intersection points of these lines do not form a blocking set in $\pi$ (see \Cref{Res:BlockingSet}).
    In particular, some line of $\pi$ doesn't contain any of these intersection points and hence contains either $\floor{\sqrt q}$ or $\floor{\sqrt q} + 1$ points of $\supp c$.
    But we have already excluded the existence of such a line since $\Deltta + 1 = \floor{\sqrt q} < \floor{\sqrt q} + 1 < q - \Deltta$.
    This yields a contradiction yet again, concluding the proof.
\end{proof}

\begin{crl}\label{crl:ThinOrThickSubspacesHYP}
    Every subspace is either thin or thick.
    In particular, every hyperplane is thin.
\end{crl}
\begin{proof}
    By \Cref{lem:ThinOrThickLinesHYP} and \Cref{thm:WeightSpectrumSubspaces} (\(r = 1\)), every subspace is either thick or thin.
    Moreover, by \Cref{lem:HalfPointTrick}, there are no thick hyperplanes.
\end{proof}

\begin{proof}[Proof of \Cref{thm:OrdinaryCase} in case $k = n-1$]
    Because we chose \(\n\) to be minimal such that \(c\) is a supposed counterexample to \Cref{thm:OrdinaryCase} (\(k=\n-1\)), we know that \(\n\geq3\) due to \Cref{eq:SzonyiWeiner} and \Cref{res:SzonyiWeiner}.
    This choice of minimality also implies --- due to \Cref{crl:ThinOrThickSubspacesHYP} and \Cref{res:Restricted} --- that the codeword \(\restr{c}{\Pi}\) is a linear combination of exactly \(\ceil{\frac{\wtInline{\restr c \Pi}}{\theta_{\n-2}}} \leq \delta\) \(\br{\n-2}\)-subspaces of \(\Pi\), for every hyperplane \(\Pi\).
    
    We know that \(\wt{c}\leq\deltta\theta_{n-1}\), otherwise $\wt c \geq \br{q-\sqrt q + 1} \theta_{n-1} > \Deltta \theta_{n-1}$ by \Cref{crl:ThinOrThickSubspacesHYP}, a contradiction.
    Hence, we may assume that \(\deltta\geq1\).
    Consider a \(\deltta\)-secant \(\ell\) to \(\supp{c}\).
    By \Cref{prp:PlaneBound}, all planes through \(\ell\) contain at least \(\deltta\br{q-\deltta+2}\) points of \(\supp{c}\), which implies that
    \begin{align}\label{eq:LowerBoundWeight}
        \wt{c}&\geq\br{\deltta\br{q-\deltta+2}-\deltta}\theta_{\n-2}+\deltta\nonumber\\
        &=\deltta q^{\n-1}-\br{\deltta^2-2\deltta}\theta_{\n-2}.
    \end{align}
    
    Let \(\Pi\) be a hyperplane containing a point of $\supp c$.
    Then $\restr c \Pi$ is a linear combination of at most $\delta$ $(n-2)$-spaces.
    Let \(\Sigma\) be one of these \(\br{\n-2}\)-spaces.
    Then $\wt{\restr c \Sigma} \geq \theta_{n-2} - (\delta-1) \theta_{n-3} > \delta \theta_{n-3}$.
    
    Now take any hyperplane $\Pi'$ through $\Sigma$.
    Then $\restr c {\Pi'}$ is a linear combination of at most $\delta$ $(n-2)$-spaces.
    Every $(n-2)$-space of $\Pi'$ not occurring in the linear combination contains at most $\delta \theta_{n-3}$ points of $\supp c$.
    Therefore, $\Sigma$ is one of the $(n-2)$-spaces occurring in the linear combination, and the points of $\supp{ \restr c {\Pi'}} \setminus \Sigma$ are contained in at most $\delta-1$ other $(n-2)$-spaces.
    
    As this holds for every hyperplane $\Pi'$ through $\Sigma$,
    \[
        \wt{c}\leq \br{q+1}\br{\deltta-1}q^{\n-2}+\theta_{\n-2}=\br{\deltta-1}q^{\n-1}+\deltta q^{\n-2}+\theta_{\n-3}.
    \]
    Combining this with \Cref{eq:LowerBoundWeight}, we obtain
    \begin{align*}
        \br{\deltta-1}q^{\n-1}+\deltta q^{\n-2}+\theta_{\n-3}&\geq\deltta q^{\n-1}-\br{\deltta^2-2\deltta}\theta_{\n-2}\\
        \iff\quad0&\geq q^{\n-1}-\br{\deltta^2-\deltta}q^{\n-2}-\br{\deltta-1}^2\theta_{\n-3}.
    \end{align*}
    Using that \(\deltta^2-\deltta<\br{\sqrt{q}-1}^2-\br{\sqrt{q}-1}\) and \(\br{\deltta-1}^2<q-1\), we get
    \[
        0>3\br{\sqrt{q}-1}q^{\n-2}+1,
    \]
    a contradiction.
\end{proof}

\subsection{The case \texorpdfstring{\(\bm{k\leq\n-2}\)}{k < n-1}}
 \label{Sec:PtK}

Due to the result of the previous subsection and the minimality of \(k\), we know that \(k\leq\n-2\).
Moreover, due to the minimality of \(\n\), combined with \Cref{res:Projection}, \(\projArg{R}{\Pi}{c}\) is a linear combination of at most \(\Deltta\) \(k\)-subspaces of \(\Pi\) for every non-incident point-hyperplane pair \(\br{R,\Pi}\), \(R\notin\supp{c}\).
Keep this observation in mind during the remainder of this subsection.

\begin{dfn}[\(k\leq\n-2\)]
    An \(i\)-space \(\iota\) is called \emph{thin} (with respect to \(c\)) in case
    \[
        \wt{\restr{c}{\iota}}\leq\Deltta\theta_{i-1}.
    \]
\end{dfn}

\begin{lem}\label{lem:ThinOrThickLines}
    Every line is either thin or thick.
\end{lem}
\begin{proof}
    Consider an arbitrary \(m\)-secant \(\ell\) with respect to \(\supp{c}\) and suppose, to the contrary, that \(\Deltta+1\leq m\leq q-\floor{\sqrt{q}}+2\).
    If every plane through \(\ell\) contained at least \(\Deltta\) points of \(\supp{c}\setminus\ell\), then
    \[
        \wt{c}\geq\Deltta\theta_{\n-2}+m>\Deltta\theta_k,
    \]
    a contradiction.
    Thus, there must exist a plane \(\pi\) through \(\ell\) containing at most \(\Deltta-1\) points of \(\supp{c}\) not lying in \(\ell\), forming a point set \(\mS\).
    
    First, assume that \(\Deltta+1\leq m\leq\frac{q}{2}\).
    Fix a point set \(\mP\) consisting of \(\Deltta+1\) points of \(\supp{\restr{c}{\ell}}\).
    By connecting each point of \(\mP\) with each point of \(\mS\), one obtains a set of lines that cover at most \(\size{\mP}\cdot\size{\mS}\cdot q\leq\br{\Deltta+1}\br{\Deltta-1}q<q^2\) points of \(\pi\setminus\ell\).
    Therefore, there exists a point \(R\in\pi\setminus\ell\) that does not lie on a line connecting a point of \(\mP\) with a point of \(\mS\).
    As a consequence, \(R\notin\supp{c}\).
    Choosing a hyperplane \(\Pi\) through \(\ell\) that does not contain \(\pi\).
    \Cref{prp:Intuition}(2) states that \(\ell\) contains either at most \(\Deltta\) or at least \(q-\Deltta+2\) points of \(\supp{\projArg{R}{\Pi}{c}}\).
    However, by the choice of the point \(R\) and the way \(\proj{R}{\Pi}\) is defined, we know that
    \begin{itemize}
        \item \(\mP\subseteq\supp{\projArg{R}{\Pi}{c}}\), implying, as \(\size{\mP}=\Deltta+1\), that \(\ell\) contains at least \(q-\Deltta+2\) points of \(\supp{\projArg{R}{\Pi}{c}}\), and
        \item \(\ell\) contains at most \(m+\size{\mS}\leq\frac{q}{2}+\Deltta-1\) points of \(\supp{\projArg{R}{\Pi}{c}}\).
    \end{itemize}
    This implies \(q-\Deltta+2\leq\frac{q}{2}+\Deltta-1\), a contradiction.

    The case \(\frac{q+1}{2}\leq m\leq q-\floor{\sqrt{q}}+2\) is similar by choosing a point set \(\mP\) consisting of \(\Deltta\) holes of \(\ell\) with respect to \(c\) and proving that \(\ell\) contains at most \(\Deltta\) but also at least \(\frac{q+1}{2}-\size{\mS}\geq\frac{q+1}{2}-\Deltta+1\) points of \(\supp{\projArg{R}{\Pi}{c}}\).
\end{proof}

\begin{crl}\label{crl:ThinOrThickSubspaces}
    Every $k$-space is thin.
\end{crl}
\begin{proof}
    This follows from \Cref{lem:ThinOrThickLines}, \Cref{thm:WeightSpectrumSubspaces} (\(r=1\)) and \Cref{lem:HalfPointTrick}.
\end{proof}

\begin{proof}[Proof of \Cref{thm:OrdinaryCase} in case $k \leq n-2$]
    Consider an \(\br{\n-k}\)-space \(\lambda\) with \(\deltta<\wt{\restr{c}{\lambda}}\leq q-\Deltta+2\).
    Then \(\wt{\restr{c}{\lambda}}\geq\Deltta+1\), so we can select a point set \(\mP\) in \(\supp{\restr{c}{\lambda}}\) of size \(\Deltta+1\).
    All lines containing at least two points of \(\mP\) cover at most \(\binom{\Deltta+1}{2}\br{q+1}<\frac{q}{2}\br{q+1}<q^2\) points of \(\lambda\setminus\supp{c}\).
    Hence, as \(\dim\brArg{\lambda}=\n-k\geq2\) and as \(\wt{\restr{c}{\lambda}}\leq q-\Deltta+2\), there must exist a point \(R\in\lambda\setminus\supp{c}\) through which each line contains at most one point of \(\mP\).
    Pick a hyperplane \(\Pi\not\ni R\).
    Then \(\projArg{R}{\Pi}{c}\) is a linear combination of at most \(\Deltta\) \(k\)-subspaces of \(\Pi\).
    Due to the choice of \(R\), at least \(\Deltta+1\) points of \(\supp{\projArg{R}{\Pi}{c}}\) lie in \(\lambda\cap\Pi\), hence at least two of these points, say \(Q_1\) and \(Q_2\), originate from the very same \(k\)-subspace of the linear combination.
    This means that the line \(\vspan{Q_1,Q_2}\) must contain at least \(q+1-\br{\Deltta-1}=q-\Deltta+2\) points of \(\supp{\projArg{R}{\Pi}{c}}\).
    By \Cref{res:Projection}(2), the plane \(\vspan{Q_1,Q_2,R}\subseteq\lambda\) must contain at least \(q-\Deltta+2\) points of \(\supp{c}\).
    
    We conclude that every \(\br{\n-k}\)-space contains either at most \(\deltta\) or at least \(q-\Deltta+2\) points of \(\supp{c}\).
    By \Cref{thm:WeightSpectrumSubspaces}, either \(\wt{c}\leq\deltta\theta_k\) or \(\wt{c}>\br{q-\sqrt{q}+1}\frac{q^\n-1}{q^{\n-k}-1}\).
    The latter implies that
    \begin{align*}
        \sqrt{q}\theta_k>\br{q-\sqrt{q}}\frac{q^\n-1}{q^{\n-k}-1}&&\Longleftrightarrow&&\frac{q^{\n+1}-q^{\n-k}-q^{k+1}+1}{q^{\n+1}-q^\n-q+1}>{\sqrt{q}-1}\\
        &&\Longrightarrow&&\frac{q^{\n+1}-2q^2+1}{q^{\n+1}-q^\n-q+1}>{\sqrt{q}-1}\\
        &&\Longleftrightarrow&&\frac{q}{q-1}-\frac{2q+1}{q^\n-1}>{\sqrt{q}-1},
    \end{align*}
    a contradiction.
    Thus, \(\wt{c}\leq\deltta\theta_k\).
    
    \parag
    Now suppose that \(\lambda\) is an \(\br{\n-k}\)-space containing precisely \(\deltta\) points of \(\supp{c}\).
    Define \(\mP\coloneqq\supp{\restr{c}{\lambda}}\), hence \(\size{\mP}=\deltta\).
    For any set $\mS$ of $x$ points in a projective space, there are at most $x + \binom x 2 (q-1)$ points lying on a line connecting two points of $\mS$.
    If $x \leq \sqrt q$, then $x + \binom x 2 (q-1) < \theta_2$.
    Since $\deltta \leq \sqrt q - 1$, there exists a point $R_1 \in \lambda$, not lying on any line connecting two points of $\mP$.
    Moreover, there exists a point $R_2$ not lying on any line connecting two points of $\mP \cup \{R_1\}$.
    
    Take a hyperplane $\Pi$ that misses $R_1$ and $R_2$.
    Let $\lambda'$ denote $\lambda \cap \Pi$.
    Note that for $i\in\set{1,2}$, $c_i \coloneqq \proj {R_i} \Pi (c)$ is a codeword of $\code k {n-1} q$ with $\wt {c_i} \leq \wt c \leq \deltta \theta_k$ and hence a linear combination of at most $\deltta$ $k$-spaces.
    In addition, since no line through $R_i$ contains more than one point of $\mP$, $\lambda'$ intersects $\supp{c_i}$ in exactly $\deltta$ points.
    It follows from \Cref{prp:Intuition}(2) that $c_i$ is a linear combination of exactly $\deltta$ $k$-spaces of $\Pi$, each intersecting $\lambda'$ exactly in a point.
    Let $\mK_i$ denote the set of these $k$-spaces.
    
    Take a \(k\)-space \(\kappa_1\in\mK_1\) and a \(k\)-space \(\kappa_2\in\mK_2\).
    Then $\vspan{R_i,\kappa_i}$ intersects $\lambda$ in a line through $R_i$ and a point of $\mP$, $i\in\set{1,2}$.
    Since the line $\vspan{R_1,R_2}$ does not contain a point of $\mP$, $\vspan{R_1,\kappa_1}$ and $\vspan{R_2,\kappa_2}$ intersect $\lambda$ in distinct lines, and in particular are distinct subspaces.
    As a result, the \(\br{k+1}\)-spaces \(\vspan{R_1,\kappa_1}\) and \(\vspan{R_2,\kappa_2}\) either intersect in a subspace of dimension at most \(k-1\) (and therefore share at most \(\theta_{k-1}\) points of \(\supp{c}\)), or intersect in a \(k\)-space.
    By \Cref{crl:ThinOrThickSubspaces}, the latter \(k\)-space must be thin, so \(\vspan{R_1,\kappa_1}\) and \(\vspan{R_2,\kappa_2}\) share at most \({\Deltta}\theta_{k-1}\) points of \(\supp{c}\).
    As a consequence, we get
    \begin{align*}
        \deltta\theta_k\geq\wt{c}&\geq\wt{\projArg{R_1}{\Pi}{c}}+\wt{\projArg{R_2}{\Pi}{c}}-\size{\mK_1}\cdot\size{\mK_2}\cdot{\Deltta}\theta_{k-1}\\
        &\geq2\cdot\deltta\br{\theta_k-\br{\deltta-1}\theta_{k-1}}-\deltta^2{\Deltta}\theta_{k-1}\\
        &=\deltta\theta_k+\br{q-2\br{\deltta-1}-\deltta{\Deltta}}\deltta\theta_{k-1}+\deltta.
    \end{align*}
    This implies that \(q-2\br{\deltta-1}-\deltta{\Deltta}<0\), or, in other words, that
    \[
        q<2\br{\deltta-1}+\deltta{\Deltta}\leq2\br{\sqrt{q}-1-1}+\br{\sqrt{q}-1}^2=q-3,
    \]
    a contradiction.
\end{proof}

\section{Codes of \texorpdfstring{\(\bm{j}\)}{j}- and \texorpdfstring{\(\bm{k}\)}{k}-spaces}
 \label{Sec:j&k}

In this section, we finish the proof of \Cref{thm:GeneralCase}.
The proof works by induction on $j$ and is very similar to the proof of our previous paper \cite[\S 6]{AdriaensenDenaux}.
Essentially, we only need to improve the lower bound of Step 1 in the proof of \cite[Theorem 6.7]{AdriaensenDenaux}.
Let us introduce the proper notation.

\begin{dfn}
\begin{enumerate}
    \item For each integer \(i\), \(0\leq i<j\), and for each $v \in \FF p^{\mG_j(n,q)}$, define $\la_i(v) \in \FF p^{\mG_i(n,q)}$ as
\[
 \la_i(v): \mG_i(n,q) \to \FF p:
 \iota \mapsto \sum_{\substack{\lambda \in \mG_j \\ \iota \subset \lambda}} v(\lambda).
\]
This means that the value of an $i$-space $\iota$ with respect to $\la_i(v)$ is the sum of the values with respect to $v$ of all $j$-spaces $\lambda$ through $\iota$.
We will denote $\la_0$ by $\la$.
    \item Define $\mK_{j,k}(n,q) \coloneqq \ker(\la_{j-1}) \cap \mC_{j,k}(n,q)$.
\end{enumerate}
\end{dfn}

We want to prove a good lower bound on the minimum weight of $\mK_{j,k}(n,q)$.
This will be done by induction on $n$.
We recall the most important properties of $\la_i$.

\begin{res}[{\cite[\S 6]{AdriaensenDenaux}}] \label{Res:La}
 Let \(i\) be an integer, \(0\leq i<j\).
 \begin{enumerate}
  \item $\la_i$ is a linear map.
  \item $\la_i$ maps $\code {j,k} n q$ to $\code {i,k} n q$, and more specifically maps $\kappa^{(j)}$ to $\kappa^{(i)}$.
  \item $\la \circ \la_i = \la$.
  \item If $c \in \mK_{j,k}(n,q)$ and $P \in \suppp 0 c$, then there are at least $2 \frac{q^{k-1}}{\theta_{j-1}} \gauss {k-1} {j-1}$ $j$-spaces of $\supp c$ incident with $P$.
  \item If $c \in \code {j,k} n q$ and $\iota \in \suppp {j-1} c$, then there are at least $\theta_{k-j}$ $j$-spaces of $\supp c$ incident with $\iota$.
 \end{enumerate}
\end{res}

We now establish a lower bound on the minimum weight of $\mK_{j,k}(n,q)$ in the base case $n = k+1$.

\begin{lem}
 \label{LmMinWtKer}
 Suppose that $q \geq 8$ and $j > 0$.
 If $c \in \mK_{j,k}(k+1,q) \setminus \{ \zero \}$, then
 \[
  \wt c > \frac 12 q \gauss{k+1}{j+1}.
 \]
\end{lem}

\begin{proof}
 Take $c \neq \zero$ in $\mK_{j,k}(k+1,q)$.
Define $T \coloneqq \supp c$ and $S \coloneqq \suppp 0 c$.
Denote their sizes by $t \coloneqq |T| = \wt c$ and $s \coloneqq |S|$.
Define $x$ as the average number of $j$-spaces $\lambda \in T$ through a point of $S$.
Then the number of incidences between $S$ and $T$ is given by
\[
 e(S,T) = t \theta_j = s x.
\]
Now apply the expander mixing lemma (\Cref{Eq:EML}).
This tells us 
\begin{align*}
 \left| t \theta_j - \frac{\theta_j}{\theta_{k+1}} st \right| < \sqrt{{q^j \gauss k j} st} &
 \implies t \theta_j \left| 1 - \frac s {\theta_{k+1}} \right| < \sqrt{{q^j \gauss k j} \frac{t^2 \theta_j} x} \\
 & \implies \left | 1 - \frac s {\theta_{k+1}} \right| < \sqrt{\frac {q^j \gauss k j} {x \theta_j}}.
\end{align*}
Since $s \leq \theta_{k+1}$, this implies that
\begin{align*}
 1 - \frac{t \theta_j}{x \theta_{k+1}} < \sqrt{\frac {q^j \gauss k j} {x \theta_j}} 
 && \implies &&
 t > \frac{x \theta_{k+1}}{\theta_j} \left( 1 - \sqrt{\frac {q^j \gauss k j} {x \theta_j}} \right).
\end{align*}
This lower bound on $t$ is increasing in $x$.
By \Cref{Res:La}(4), $x \geq 2 \frac{q^{k-1}}{\theta_{j-1}} \gauss{k-1}{j-1}$, so we get that
\begin{align}
 \label{EqLowerBoundT}
 t > 2 \frac{q^{k-1}}{\theta_{j-1}} \gauss{k-1}{j-1} \frac{\theta_{k+1}}{\theta_j} \left( 1 - \sqrt{\frac{q^j \gauss k j} {{2 \frac{q^{k-1}}{\theta_{j-1}} \gauss{k-1}{j-1} \theta_j}}} \right).
\end{align}

The rest of the proof consists of estimations of the above expression.
First, consider the expression under the square root.
\begin{align*}
 \frac 12 \frac{q^j}{q^{k-1}} \frac{\gauss k j}{\gauss {k-1}{j-1}} \frac 1 {\frac{\theta_j}{\theta_{j-1}}}
 &= \frac 12 \frac{q^j}{q^{k-1}} \frac{q^k-1}{q^j-1} \frac{q^j-1}{q^{j+1}-1}
 = \frac 12 \frac{q^{k+j}-q^j}{q^{k+j}-q^{k-1}} \\
 &< \frac 12 \frac{q^{k+j}}{q^{k+j}-q^{k+j-2}}
 = \frac 1 2 \frac{q^2}{q^2-1}.
\end{align*}
Next, we prove that
\[
 \frac{q^{k-1}}{\theta_{j-1}} \gauss{k-1}{j-1} \frac{\theta_{k+1}}{\theta_j} > (q-1) \gauss{k+1}{j+1}.
\]
Dividing both sides by $\gauss {k-1}{j-1}$, we need to prove that
\[
 \frac{q^{k-1}(q-1)}{q^j-1} \frac{q^{k+2}-1}{q^{j+1}-1} > (q-1) \frac{q^{k+1}-1}{q^{j+1}-1} \frac{q^k-1}{q^j-1}.
\]
This is equivalent to
\[
 q^{k-1}(q^{k+2}-1) > (q^{k+1}-1)(q^k-1),
\]
which is easy to check.

Combining this with \Cref{EqLowerBoundT}, this yields
\[
 \wt c > \underbrace{2 \br{1 - \frac 1 q} \left( 1 - \sqrt{\frac 12 \frac{q^2}{q^2-1}} \right)}_{\eqqcolon C_q} q \gauss{k+1}{j+1}
\]
The expression $C_q$ is increasing in $q$.
Since we assumed that $q \geq 8$, we get that $C_q \geq C_8 \geq \frac 12$, the last inequality being checked by computer.
\end{proof}

For the induction step, we will again make use of $\prj j R \Pi$ (see \Cref{dfn:Proj}).

\begin{lem} \label{LmProj}
 Consider a point $R$ and a hyperplane $\Pi \not \ni R$.
 \begin{enumerate}
     \item For $i<j$, $\la_i \circ \prj j R \Pi = \prj i R \Pi \circ \la_i$.
     \item $\prj j R \Pi$ maps $\mK_{j,k}(n,q)$ to $\mK_{j,k}(n-1,q)$.
 \end{enumerate}
\end{lem}

\begin{proof}
 (1) 
 Take $v \in \FF p^{\mG_j(n,q)}$.
 Choose an $i$-space $\iota$ in $\Pi$.
 Then
 \begin{align*}
  \la_i \br{ \prj j R \Pi (v) } (\iota)
  &= \sum_{\substack{\lambda' \in \mG_j(\Pi) \\ \iota \subset \lambda' }} \prj j R \Pi (v) (\lambda')
  = \sum_{\substack{\lambda' \in \mG_j(\Pi) \\ \iota \subset \lambda' }} \sum_{\lambda \in \mG_j(\vspan{R,\lambda'})} v(\lambda) \\
  &= \sum_{\lambda \in \mG_j(n,q)} v(\lambda) \underbrace{\left| \sett{\lambda' \in \mG_j(\Pi)}{\iota \subset \lambda', \, \lambda \subset \vspan{R,\lambda'}} \right|}_{\eqqcolon f_1(\lambda)}.
 \end{align*}
 On the other hand,
 \begin{align*}
  \prj i R \Pi \br{\la_i(v)}(\iota)
  & = \sum_{\iota' \in \mG_i(\vspan{R,\iota})} \la_i(v) (\iota')
  = \sum_{\iota' \in \mG_i(\vspan{R,\iota})} \sum_{\substack{\lambda \in \mG_j(n,q) \\ \iota' \subset \lambda}} v(\lambda) \\
  & = \sum_{\lambda \in \mG_j(n,q)} v(\lambda) \underbrace{|\sett{\iota' \in \mG_i(\vspan{R,\iota})}{ \iota' \subset \lambda }|}_{\eqqcolon f_2(\lambda)}.
 \end{align*}
 
 Thus, we need to prove that $f_1(\lambda) \equiv f_2(\lambda) \pmod p$, for every $j$-space $\lambda$.
 Note that 
 \[
  f_2(\lambda) = |\mG_i(\vspan{R,\iota} \cap \lambda)|
  \equiv \begin{cases}
   1 & \text{if } \dim \br{\vspan{R,\iota} \cap \lambda} \geq i, \\
   0 & \text{otherwise}
  \end{cases}
  \pmod p.
 \]
 Moreover,
 \begin{align*}
  \dim \br{\vspan{R,\iota} \cap \lambda} = 
  \begin{cases}
   \dim \br{\iota \cap \lambda} + 1 & \text{if } R \in \lambda, \\
   \dim \br{\iota \cap \lambda}     & \text{if } R \notin \lambda.
  \end{cases}
 \end{align*}
 
 On the other hand,
 \begin{align*}
  f_1(\lambda) 
  &= |\sett{\lambda' \in \mG_j(\Pi)}{\iota \subset \lambda', \vspan{R,\lambda} \cap \Pi \subseteq \lambda'}| \\
  &\equiv \begin{cases}
   1 & \text{if } \dim\br{\vspan{\iota,\vspan{R,\lambda} \cap \Pi}} \leq j, \\
   0 & \text{otherwise}
  \end{cases}
  \pmod p.
 \end{align*}
 Moreover, 
 \begin{align*}
  \dim\br{\vspan{\iota,\vspan{R,\lambda} \cap \Pi}}
  = \dim \br{\vspan{\iota,R,\lambda}} - 1
  = \begin{cases}
   \dim \br{\vspan{\iota,\lambda}} - 1 & \text{if } R \in \lambda, \\
   \dim \br{\vspan{\iota,\lambda}} & \text{if } R \notin \lambda.
  \end{cases}
 \end{align*}

 Thus, we need to prove that
 \[
  \begin{cases}
   \dim \br{\iota \cap \lambda} + 1 \geq i \iff \dim \br{\vspan{\iota,\lambda}} - 1 \leq j & \text{if } R \in \lambda, \\
   \dim \br{\iota \cap \lambda} \geq i \iff \dim \br{\vspan{\iota,\lambda}} \leq j & \text{if } R \notin \lambda.
  \end{cases}
 \]
 This follows in both cases from Grassmann's identity:
 \(
  \dim \br{\iota \cap \lambda} + \dim \br{\vspan{\iota,\lambda}} = i + j.
 \)

 \bigskip
 
 (2) It follows directly from (1) that $\prj j R \Pi$ maps $\ker(\la_{j-1})$ to $\ker(\la_{j-1})$.
 By \cref{res:Projection} (1), $\prj j R \Pi$ also maps $\code {j,k} n q$ to $\code {j,k} {n-1} q$.
 The statement follows.
\end{proof}

\begin{prp}
 \label{PropMinWtKer}
 Suppose that $q \geq 8$ and $j > 0$.
 If $c \in \mK_{j,k}(n,q) \setminus \{ \zero \}$, then
 \[
  \wt c > \frac 12 q \gauss{k+1}{j+1}.
 \]
\end{prp}

\begin{proof}
 We prove this proposition by induction on $n$.
 The base case $n=k+1$ was dealt with in \Cref{LmMinWtKer}.
 So assume that $n \geq k+2$ and suppose that the proposition holds for $\mK_{j,k}(n-1,q)$.
 Take a non-zero codeword $c \in \mK_{j,k}(n,q)$ and a $j$-space $\lambda \in \supp c$.
 We again denote $\suppp 0 c$ by $S$.
 
 First, suppose that there exists a $(j+1)$-space $\sigma$ through $\lambda$ that contains no other $j$-space of $\supp c$ and contains a point $R \notin S$.
 Choose a hyperplane $\Pi$ intersecting $\sigma$ in $\lambda$.
 By \Cref{LmProj}(2), $\prj j R \Pi (c)$ is a codeword of $\mK_{j,k}(n-1,q)$.
 Moreover, since $\lambda$ is the only $j$-space of $\supp c$ in $\sigma$, $\prj j R \Pi (c)(\lambda) = c(\lambda) \neq 0$.
 In particular, this means that $\prj j R \Pi (c) \neq \zero$.
 Using \Cref{res:Projection}(2) and the induction hypothesis, this implies that
 \[
  \wt c \geq \wt{\prj j R \Pi (c)} > \frac 12 q \gauss{k+1}{j+1}.
 \]

 Now suppose that every $(j+1)$-space through $\lambda$ contains either another $j$-space of $\supp c$ or contains no points outside of $S$.
 Then every $(j+1)$-space through $\lambda$ contains at least $q^j$ points of $S \setminus \lambda$.
 Therefore,
 \[
  |S| \geq \theta_j + \theta_{n-j-1} q^j > \theta_{n-1} \geq \theta_{k+1}.
 \]
 As in the proof of \Cref{LmMinWtKer}, we have that
 \[
  \wt c 
  \geq 2 \frac{q^{k-1} \gauss{k-1}{j-1}}{\theta_{j-1} \theta_j} |S|
  \geq 2 \frac{q^{k-1} \gauss{k-1}{j-1}}{\theta_{j-1} \theta_j} \theta_{k+1}.
 \]
 It suffices to check that the right-hand side of the above inequality is greater than $\frac 12 q \gauss{k+1}{j+1}$.
 This is equivalent to
 \[
  2 q^{k-1} \br{q^{k+2}-1} (q-1) > \frac 1 2 q \br{q^{k+1}-1} \br{q^k-1}.
 \]
 Since $q^{k+2}-1 > q \br{q^{k+1}-1}$, this follows from
 \[
  2 q^{k-1}(q-1) > \frac 1 2 \br{q^k-1},
 \]
 which can be easily checked to hold for every $q \geq 2$.
\end{proof}

\begin{prp} \label{Prop:InductionOnJ}
 Suppose that $q \geq 8$ and suppose that $C_q \leq \frac 1 4 q$ is a constant such that every codeword $c \in \mC_{j-1,k}(n,q)$ with $\wt c \leq C_q \gauss{k+1}j$ is a linear combination of exactly $\ceil{\frac{\wt c}{\gauss{k+1}{j}}}$ $k$-spaces.
 Then every codeword $c \in \mC_{j,k}(n,q)$ with $\wt c \leq C_q \gauss{k+1}{j+1}$ is a linear combination of exactly $\ceil{\frac{\wt c}{\gauss{k+1}{j+1}}}$ $k$-spaces.
\end{prp}

\begin{proof}
Suppose that the condition of the proposition holds for $\mC_{j-1,k}(n,q)$.
Take a codeword $c \in \mC_{j,k}(n,q)$ with $\wt c \leq C_q \gauss{k+1}{j+1}$.
Consider $c' \coloneqq \la_{j-1}(c)$.
Perform a double count on
\[
 \sett{(\iota,\lambda) \in \supp{c'} \times \supp c}{\iota \subset \lambda}.
\]
If $\iota \in \supp {c'}$, then $\iota$ is incident with at least $\theta_{k-j}$ $j$-spaces of $\supp c$ by \Cref{Res:La}(5).
Hence,
\[
 \wt{c'} \leq \wt{c} \frac{\theta_{j} }{\theta_{k-j} }
 \leq C_q \frac{q^{j+1}-1}{q^{k-j+1}-1} \gauss{k+1}{j+1} = C_q \gauss{k+1}{j}.
\]
By our hypothesis, $c'$ is a linear combination of at most $C_q$ $k$-spaces, i.e.\
\[
 c' = \sum_{i=1}^{C_q} \alpha_i \kappa_i^{(j-1)},
\]
for some scalars $\alpha_i$ and $k$-spaces $\kappa_i$.
Then 
\[ 
 c' = \la_{j-1} \left( \sum_{i=1}^{C_q} \alpha_i \kappa_i^{(j)} \right)
\]
by \Cref{Res:La}(1,2).

This implies that $c'' \coloneqq c - \sum_{i=1}^{C_q} \alpha_i \kappa_i^{(j)}$ is contained in $\mK_{j,k}(n,q)$.
Since $c''$ is equal to the difference of two functions, each of weight at most $C_q \gauss{k+1}{j+1}$, we find that $\wt{c''} \leq 2 C_q \gauss{k+1}{j+1} \leq \frac 12 q \gauss{k+1}{j+1}$.
By \Cref{PropMinWtKer}, this means that $c'' = \zero$, hence $c$ is a linear combination of at most $C_q$ characteristic functions of $k$-spaces.
By \Cref{prp:Intuition}(1), $c$ is a linear combination of exactly $\ceil{\wt c / \gauss{k+1}{j+1}}$ $k$-spaces.
\end{proof}

\Cref{thm:GeneralCase} now follows immediately by inductively applying \Cref{Prop:InductionOnJ}, with \Cref{thm:OrdinaryCase} as base case.
One only needs to check that $\Deltta \leq \frac 1 4 q$, which follows directly from $\Deltta < \sqrt q$ and $q \geq 32$.

\printbibliography

\noindent \begin{tabular}{l l}
 Sam Adriaensen &
 Lins Denaux \\
 \textit{Vrije Universiteit Brussel} &
 \textit{Ghent University} \\
 Department of Mathematics &
 Department of Mathematics: Analysis, \\
 \; and Data Science 
 &  \; Logic and Discrete Mathematics \\
 Pleinlaan 2 -- Building G &
 Krijgslaan \(281\) -- Building S\(8\)\\
 \(1050\) Elsene &
 \(9000\) Ghent \\
 BELGIUM &
 BELGIUM \\
 \texttt{e-mail:} \href{mailto:sam.adriaensen@vub.be}{\texttt{sam.adriaensen@vub.be}} &
 \texttt{e-mail:} \href{mailto:lins.denaux@ugent.be}{\texttt{lins.denaux@ugent.be}} \\
 \texttt{website:} \href{https://samadriaensen.wordpress.com/}{\texttt{samadriaensen.wordpress.com}}
 & \texttt{website:} \href{https://cage.ugent.be/~ldnaux}{\texttt{cage.ugent.be/\raisebox{0.5ex}{\texttildelow}ldnaux}}
\end{tabular}

\end{document}